\documentclass[12pt]{amsart}
\usepackage{latexsym}

\usepackage{amssymb}
\usepackage{bm}
\usepackage{graphicx}
\usepackage{psfrag}
\usepackage{color}
\usepackage{hyperref}
\hypersetup{colorlinks=true, linkcolor=blue, citecolor=red, urlcolor=wine}
\usepackage{url}
\usepackage{algpseudocode}
\usepackage{fancyhdr}
\usepackage{xy}
\input xy
\xyoption{all}
\usepackage{stmaryrd}

\newtheorem*{maintheorem*}{Main Theorem}
\newtheorem{theorem}{Theorem}[section]
\newtheorem{prop}[theorem]{Proposition}

\newtheorem{lemma}[theorem]{Lemma}
\newtheorem{cor}[theorem]{Corollary}

\theoremstyle{definition}
\newtheorem{definition}[theorem]{Definition}
\newtheorem{remark}[theorem]{Remark}
\newtheorem{example}[theorem]{Example}
\numberwithin{equation}{section}

\newcommand{\nn}{\mathbb{N}}
\newcommand{\pp}{\mathbb{P}}
\newcommand{\qq}{\mathbb{Q}}
\newcommand{\rr}{\mathbb{R}}
\newcommand{\zz}{\mathbb{Z}}

\providecommand\ldb{\llbracket}
\providecommand\rdb{\rrbracket}
\newcommand\pval{\mathsf{v}_p}
\newcommand{\gp}{\textsf{gp}}



\begin{document}

\title{When Is a Puiseux Monoid Atomic?}
 
\author{Scott T. Chapman}
\address{Department of Mathematics\\Sam Houston State University\\Huntsville, TX 77341}
\email{scott.chapman@shsu.edu}


\author{Felix Gotti}
\address{Department of Mathematics\\University of Florida\\Gainesville, FL 32611}
\email{felixgotti@ufl.edu}

\author{Marly Gotti}
\address{Research and Development \\ Biogen \\ Cambridge \\ MA 02142}
\email{marly.cormar@biogen.com}

\maketitle

\date{\today}
	
	\begin{abstract}
		A Puiseux monoid is an additive submonoid of the nonnegative rational numbers. If $M$ is a Puiseux monoid, then the question of whether each non-invertible element of $M$ can be written as a sum of irreducible elements (that is, $M$ is atomic) is surprisingly difficult. Although various techniques have been developed over the past few years to identify subclasses of Puiseux monoids that are atomic, no general characterization of such monoids is known. Here we survey some of the most relevant aspects related to the atomicity of Puiseux monoids. We provide characterizations of when $M$ is finitely generated, factorial, half-factorial, other-half-factorial, Pr\"ufer, seminormal, root-closed, and completely integrally closed. In addition to the atomicity, characterizations are also not known for when $M$ satisfies the ACCP, the bounded factorization property, or the finite factorization property. In each of these cases, we construct an infinite class of Puiseux monoids satisfying the corresponding property.
	\end{abstract}


\medskip
\section{Introduction}
\label{sec:intro}

In a standard undergraduate abstract algebra course, or in a beginning graduate algebra course, the study of monoids plays a supporting role to the study of groups, rings, modules, and fields.  If you are interested in studying monoids, then where should you begin?  Since most group and ring curricula begin with the study of the integers, the study of additive or multiplicative submonoids of $\mathbb{Z}$ might be a logical candidate.  While the properties of multiplicative submonoids of $\mathbb{Z}$ are probably partially covered in a beginning number theory course, the additive properties of (cofinite) submonoids of the nonnegative integers, called \textit{numerical monoids}, are normally relegated to those interested in the topic for research purposes.  This is disappointing since many of the interesting aspects of these monoids (such as the Frobenius number or the Ap\'{e}ry set) can be easily comprehended by an advanced undergraduate.  The same is true to a far lesser extent for additive submonoids of the nonnegative rational numbers (which we denote by $\mathbb{Q}_{\geq 0}$).  Such submonoids pose much deeper algebraic questions, as opposed to their numerical monoid cousins.

The additive submonoids of $\qq_{\ge 0}$ naturally appear in commutative ring theory as valuations of subdomains of the field of Puiseux series, which was first studied by the French mathematician Victor A. Puiseux back in 1850~\cite{vP1850}. Given this connection and to honor Victor Puiseux, additive submonoids of $\qq_{\ge 0}$ have been investigated under the term \emph{Puiseux monoids} (see~\cite[Section~3]{fG17} for more details on this). Here we explore Puiseux monoids in connection with the following basic question: when can every element of a Puiseux monoid be written as a sum of irreducible elements? In general, a monoid where each element has this property is known as an \emph{atomic} monoid.  

Why is the atomic property interesting?  In general, atomic integral domains are barely mentioned in a beginning course of algebra. Clearly, the property of being atomic is a natural relaxation of that of being a UFD. Many relevant classes of integral domains consist of members that are atomic but not UFDs, including those of Dedekind domains (in particular, rings of algebraic integers) and, more generally, Noetherian domains and Krull domains. For some of these classes, the atomic structure of their members determines, up to a certain extent, some of their algebraic properties. For instance, it was proved by Carlitz~\cite{lC60} that a ring of algebraic integers is half-factorial (i.e., the lengths of any two irreducible factorizations of an element are equal) if and only if the size of its class group is at most $2$. The atomic structure of many classes of integral domains has been a popular topic in the mathematical literature over the past 30 years (see~\cite{GH06} and the references therein).

Why do we narrow this investigation down to monoids?  It turns out that several problems involving factorizations of elements in an integral domain depend solely on its multiplicative structure. As a result, many of the papers written over the past three decades treat the phenomenon of nonunique factorizations in the context of monoids. In particular,~\cite{BCKR} and~\cite{CHM06} pioneered the study of factorization problems in numerical monoids. Since then, many papers have been dedicated to the atomicity and factorization of numerical monoids~\cite{ACHP07,CCMMP17,CHK09,GS18} and some of their higher-rank generalizations~\cite{CKO02,CO17,GOS13,fG20}. More recently, a systematic investigation of the atomicity of Puiseux monoids has been initiated~(see~\cite{CGG19,fG19a} and references therein).

Why do we study Puiseux monoids?  Although the atomicity of Puiseux monoids has earned attention only in the last few years, since the 1970s Puiseux monoids have been crucial in the construction of various significant examples in commutative ring theory. In 1974, Grams~\cite{aG74} used an atomic Puiseux monoid as the main ingredient to construct the first example of an atomic integral domain that does not satisfy the ascending chain condition on principal ideals (or ACCP); this refuted Cohn's assertion that every atomic integral domain satisfies the ACCP \cite[Proposition 1.1]{PMC}. In addition, in~\cite{AAZ90}, Anderson, Anderson, and Zafrullah appealed to Puiseux monoids to construct various needed examples of integral domains satisfying certain prescribed properties. More recently, Puiseux monoids have played an important role in~\cite{CG19}, where Coykendall and the second author partially answered a question on the atomicity of monoid rings posed by Gilmer in the 1980s (see~\cite[p. 189]{rG84}).

Puiseux monoids have also been important in factorization theory. For instance, the class of Puiseux monoids comprises the first (and only) examples known so far of primary atomic monoids with irrational elasticity (this class was found in \cite[Section~4]{GGT19} via \cite[Theorem~3.2]{GO17}). A Puiseux monoid is a suitable additive structure containing simultaneously several copies of numerical monoids independently generated. This fact has been harnessed by Geroldinger and Schmid \cite[Theorem~3.3]{GS18} to achieve a nice realization theorem for the sets of lengths of numerical monoids (the \textsc{Monthly} article~\cite{aG16} introduces readers to sets of lengths).

Given the relevance of the class of Puiseux monoids in both commutative algebra and factorization theory, as well as the substantial attention they have received in the last few years, we compile here some of the most significant developments concerning their atomicity. In Section~\ref{sec:prelim}, we review the basic notation and definitions necessary for the remainder of our work. In Section~\ref{sec:closure nad conductor} we show that there are uncountably many Puiseux monoids up to isomorphism, and that they make up the class of all additive submonoids of $\qq$ that are not groups (Propositions~\ref{Gilmer} and~\ref{prop:there are uncountably many PMs}). In addition, we describe both the root and the complete integral closures as well as the conductor of a Puiseux monoid. Section~\ref{sec:atomicity} delves into issues related to atomicity. For each atomic property in the well-known diagram
\begin{equation*} 
	\textbf{UFM} \ \Rightarrow \ \textbf{FFM} \ \Rightarrow \ \textbf{BFM} \ \Rightarrow \ \textbf{ACCP} \ \Rightarrow \ \textbf{atomic monoid}
\end{equation*}
we present an infinite class of Puiseux monoids satisfying such a property (the atomic classes involved in the diagram are defined in Section~\ref{sec:prelim}). In addition, we provide examples testifying that none of the four implications in the above implication diagram is reversible in the class of Puiseux monoids. Finally, we show that the half-factorial and the unique factorization properties are equivalent for Puiseux monoids, and we characterize them.



Before moving forward, we highlight some connections and references related to previous studies of Puiseux monoids that we will not pursue here. The elasticity of Puiseux monoids has been studied in~\cite{GO17,mG19} while their systems of sets of lengths have received some attention in~\cite{fG19a}. In addition, factorization invariants of Puiseux monoids and numerical monoids have been compared and contrasted in~\cite{CGG19}. Transfer homomorphisms from Puiseux monoids were studied in~\cite{fG18} and monoid algebras of Puiseux monoids have been considered in~\cite{ACHZ07,CG19,GK20,fG19d}. Finally, some connections between Puiseux monoids and music theory have been recently highlighted by Bras-Amor\'{o}s in~\cite{mBA19}.

\medskip
\section{Preliminary}
\label{sec:prelim}

In this section, we introduce the relevant concepts of commutative monoids and factorization theory required to follow our exposition. General references for background information can be found in~\cite{pG01} for commutative monoids and in~\cite{GH06} for atomic monoids and factorization theory.

\smallskip
\subsection{General Notation}

Here we let $\nn$ denote the set of positive integers while we set $\nn_0 := \nn \cup \{0\}$. Also, we let $\pp$ denote the set of all prime numbers. For $a,b \in \zz$ we let $\ldb a,b \rdb$ denote the set of integers between $a$ and $b$, i.e., 
\[
	\ldb a,b \rdb := \{z \in \zz \mid a \le z \le b\}.
\]
Note that $\ldb a,b \rdb$ is empty provided that $a > b$. For $X \subseteq \rr$ and $r \in \rr$, we set
\[
	X_{\ge r} := \{x \in X \mid x \ge r\}
\]
and we use the notations $X_{> r}, X_{\le r}$, and $X_{< r}$ in a similar way. If $q \in \qq \setminus \{0\}$, then we call the unique $n \in \zz$ and $d \in \nn$ such that $q = n/d$ and $\gcd(n,d)=1$ the \emph{numerator} and \emph{denominator} of $q$ and denote them by $\mathsf{n}(q)$ and $\mathsf{d}(q)$, respectively. Finally, for $Q \subseteq \qq \setminus \{0\}$, we set
\[
	\mathsf{n}(Q) := \{\mathsf{n}(q) \mid q \in Q\} \quad \text{ and } \quad \mathsf{d}(Q) := \{\mathsf{d}(q) \mid q \in Q\}.
\]

\smallskip
\subsection{Commutative Monoids}

Throughout this article, the term \emph{monoid} stands for a commutative and cancellative semigroup with identity. Unless we specify otherwise, monoids are written additively, with identity element~$0$. Let $M$ be a monoid. We let~$M^\bullet$ denote the set $M \setminus \{0\}$ while we let $U(M)$ denote the set of invertible elements of~$M$. When $M^\bullet = \emptyset$ we say that $M$ is \emph{trivial} and when $U(M) = \{0\}$ we say that~$M$ is \emph{reduced}.

For $S \subseteq M$, we let $\langle S \rangle$ denote the smallest (under inclusion) submonoid of $M$ containing $S$, i.e., the submonoid of $M$ generated by $S$. The monoid $M$ is \emph{finitely generated} if $M$ can be generated by a finite set. An element $a \in M \setminus U(M)$ is an \emph{atom} provided that the equality $a = x+y$ for $x,y \in M$ implies that either $x \in U(M)$ or $y \in U(M)$. The set of atoms of $M$ is denoted by $\mathcal{A}(M)$. The monoid $M$ is \emph{atomic} if each element in $M \setminus U(M)$ can be written as a sum of atoms. By \cite[Proposition~2.7.8(4)]{GH06}, every finitely generated monoid is atomic. On the other hand, $M$ is \emph{antimatter} if $\mathcal{A}(M) = \emptyset$.

A subset $I$ of $M$ is called an \emph{ideal} of $M$ provided that $I + M \subseteq I$ (or, equivalently, $I + M = I$). An ideal $I$ is said to be \emph{principal} if $I = x + M$ for some $x \in M$, and~$M$ satisfies the \emph{ascending chain condition on principal ideals} (or \emph{ACCP}) provided that every increasing sequence of principal ideals of $M$ eventually stabilizes. It is well known (and not hard to argue) that every monoid satisfying the ACCP must be atomic \cite[Proposition~1.1.4]{GH06}.

An equivalence relation $\rho \subseteq M \times M$ is a \emph{congruence} if it is compatible with the operation of the monoid $M$, i.e., for all $x,y,z \in M$ with $(x,y) \in \rho$ it follows that $(z+x,z+y) \in \rho$. It can be readily verified that the set $M/\rho$ consisting of the equivalence classes of a congruence $\rho$ is a commutative semigroup with identity. For $x,y \in M$, we say that $x$ \emph{divides} $y$ \emph{in} $M$ and write $x \mid_M y$ provided that $x + x' = y$ for some $x' \in M$. Two elements $x,y \in M$ are \emph{associates} if $y = x + u$ for some $u \in U(M)$. Being associates defines a congruence on $M$ whose semigroup of classes is a monoid; this monoid is usually denoted by $M_{\text{red}}$. Observe that when $M$ is reduced one can identify $M_{\text{red}}$ with $M$.

Given a monoid $M$ we can form the group $\gp(M) := \{x-y \mid x,y \in M\}$ much in the same way that the integers are constructed from the natural numbers. Here $\gp(M)$ is an abelian group (unique up to isomorphism) satisfying the property that any abelian group containing a homomorphic image of $M$ also contains a homomorphic image of $\gp(M)$. We call $\gp(M)$ the \emph{difference group}\footnote{The difference group is called the quotient group when monoids are written multiplicatively.} of $M$. The \emph{rank} of a monoid $M$ is the rank of $\gp(M)$ when viewed as a $\mathbb{Z}$-module (i.e., the size of the largest linearly independent set in $\gp(M)$ over $\mathbb{Z}$).  

A \emph{numerical monoid} is a submonoid $N$ of $(\nn_0,+)$ that satisfies $|\nn_0 \setminus N| < \infty$. If $N \neq \nn_0$, then $\max (\nn_0 \! \setminus \! N)$ is called the \emph{Frobenius number} of $N$. Numerical monoids are finitely generated and, therefore, atomic with finitely many atoms. The \emph{embedding dimension} of $N$ is the cardinality of $\mathcal{A}(N)$. For an introduction to numerical monoids, see \cite{GR09}, and for some of their many applications, see~\cite{AG16}.

\smallskip
\subsection{Factorizations}

A multiplicative monoid $F$ is called \emph{free with basis}~$P \subset F$ if every element $x \in F$ can be written uniquely in the form
\[
	x = \prod_{p \in P} p^{\pval(x)},
\]
where $\pval(x) \in \nn_0$ and $\pval(x) > 0$ only for finitely many elements $p \in P$. By the fundamental theorem of arithmetic, the multiplicative monoid $\nn$ is free on $\pp$. In this case, we can extend $\pval$ to $\qq_{\ge 0}$ as follows. For $r \in \qq_{> 0}$ let $\pval(r) := \pval(\mathsf{n}(r)) - \pval(\mathsf{d}(r))$ and set $\pval(0) = \infty$. The maps $\pval$ (for $p \in \pp$), usually called \emph{$p$-adic valuations}, are useful tools to study Puiseux monoids.

Let $M$ be a reduced monoid. The \emph{factorization monoid} of $M$, denoted by $\mathsf{Z}(M)$, is the free (commutative) monoid on $\mathcal{A}(M)$. The elements of $\mathsf{Z}(M)$ are called \emph{factorizations}. If $z = a_1 \cdots a_n \in \mathsf{Z}(M)$, where $a_1, \dots, a_n \in \mathcal{A}(M)$, then $|z| := n$ is the \emph{length} of~$z$. The unique monoid homomorphism $\pi \colon \mathsf{Z}(M) \to M$ satisfying $\pi(a) = a$ for all $a \in \mathcal{A}(M)$ is the \emph{factorization homomorphism} of~$M$. For each $x \in M$,
\[
	\mathsf{Z}(x) := \pi^{-1}(x) \subseteq \mathsf{Z}(M) \quad \text{and} \quad \mathsf{L}(x) := \{|z| : z \in \mathsf{Z}(x)\}
\]
are the \emph{set of factorizations} and the \emph{set of lengths} of $x$, respectively. Factorization invariants stemming from the sets of lengths have been studied for several classes of atomic monoids and domains; see, for instance, \cite{CGGR01,CGTV16,CGP14,CHM06}. In particular, the sets of lengths of numerical monoids have been studied in~\cite{ACHP07,CDHK10,GS18}. In~\cite{GS18} the sets of lengths of numerical monoids were studied using techniques involving Puiseux monoids. An overview of sets of lengths and the role they play in factorization theory can be found in~\cite{aG16}. 

By restricting the size of the sets of factorizations/lengths, one obtains subclasses of atomic monoids that have been systematically studied by many authors. We say that a reduced atomic monoid $M$ is
\begin{enumerate}
	\item a \emph{UFM} (or a \emph{factorial monoid}) if $|\mathsf{Z}(x)| = 1$ for all $x \in M$,
	\vspace{2pt}
	
	\item an \emph{HFM} (or a \emph{half-factorial monoid}) if $|\mathsf{L}(x)| = 1$ for all $x \in M$,
	\vspace{2pt}
	
	\item an \emph{FFM} (or a \emph{finite factorization monoid}) if $|\mathsf{Z}(x)| < \infty$ for all $x \in M$, and
	\vspace{2pt}

	\item a \emph{BFM} (or a \emph{bounded factorization monoid}) if $|\mathsf{L}(x)| < \infty$ for all $x \in M$.
\end{enumerate}

\medskip
\section{Closures and Conductor}
\label{sec:closure nad conductor}

In this section we study some algebraic aspects of Puiseux monoids.

\begin{definition}
	A \emph{Puiseux monoid} is an additive submonoid of $\qq_{\ge 0}$.
\end{definition}

Clearly, Puiseux monoids are natural generalizations of numerical monoids. As with numerical monoids, Puiseux monoids are reduced. However, as we shall see later, Puiseux monoids are not, in general, finitely generated or atomic. Puiseux monoids account up to isomorphism for all submonoids of $(\qq,+)$ which are not groups. The following proposition sheds some light upon this observation. 

\begin{prop}\label{Gilmer}
		Let $M$ be an additive submonoid of $\mathbb{Q}$ that is not a group. Then either $M$ or $-M$ is a Puiseux monoid.
\end{prop}

\begin{proof}
	Let $M$ be as in the hypothesis and suppose, by way of contradiction, that neither~$M$ nor $-M$ is a Puiseux monoid. Fix $x \in M^\bullet$. Since neither~$M \subseteq \qq_{\ge 0}$ nor $-M \subseteq \qq_{\ge 0}$, there exists $y \in M$ such that $xy < 0$. Choose $\alpha \in \nn$ such that $ \mathsf{d}(x) \mathsf{n}(y) \mid (-1-\alpha)$.  Set
	\[
		\beta := \frac{(-1-\alpha) \mathsf{n}(x) \mathsf{d}(y)}{\mathsf{d}(x) \mathsf{n}(y)}= \frac{(-1-\alpha)x}{y}. 
	\]
	By construction, $\beta\in \mathbb{Z}$. Because $(-1-\alpha) < 0$ and $x/y < 0$, the integer $\beta$ is positive. Thus,
	\[
		-x = \alpha x + (-1-\alpha)x = \alpha x +\beta y \in M.
	\]
	As each element of $M$ is invertible, $M$ is a group. However, this contradicts the hypothesis of the proposition.
\end{proof}

\smallskip
\subsection{The Difference Group}

A monoid $M$ is \emph{torsion-free} if for all $x,y \in M$ and $n \in \nn$, the equality $nx = ny$ implies that $x=y$. Clearly, $M$ is a torsion-free monoid if and only if $\gp(M)$ is a torsion-free group. Each Puiseux monoid $M$ is obviously torsion-free and, therefore, $\gp(M)$ is a torsion-free group. Moreover, for a Puiseux monoid $M$, one can take the difference group $\gp(M)$ to be a subgroup of $(\qq,+)$, specifically,
\begin{equation} \label{eq:difference group of a PM}
	\gp(M) = \{ x-y \mid x,y \in M \}.
\end{equation}

Puiseux monoids can be characterized as follows.

\begin{prop} \label{prop:characterization property of Puiseux monoids}
	For a nontrivial monoid $M$ the following statements are equivalent.
	\begin{enumerate}
		\item $M$ is a rank-$1$ torsion-free monoid that is not a group.
		\vspace{2pt}
		
		\item $M$ is isomorphic to a Puiseux monoid.
	\end{enumerate}
\end{prop}

\begin{proof}
	To argue (1) $\Rightarrow$ (2), first note that $\gp(M)$ is a rank-$1$ torsion-free abelian group. Therefore it follows from \cite[Section 85]{lF73} that $\gp(M)$ is isomorphic to a subgroup of $(\qq,+)$, and one can assume that $M$ is a submonoid of $(\qq,+)$. Since $M$ is not a group, Proposition \ref{Gilmer} ensures that either $M \subseteq \qq_{\le 0}$ or $M \subseteq \qq_{\ge 0}$. So $M$ is isomorphic to a Puiseux monoid. To verify (2) $\Rightarrow$ (1), let us assume that $M \subseteq \gp(M) \subseteq \qq$. As $\gp(M)$ is a subgroup of $(\qq,+)$, it is a rank-$1$ torsion-free abelian group. This implies that~$M$ is a rank-$1$ torsion-free monoid. Since $M$ is nontrivial and reduced, it cannot be a group, which completes our proof.
\end{proof}

While a Puiseux monoid is countable, the class of Puiseux monoids is uncountable, as the next proposition illustrates.

\begin{prop} \label{prop:there are uncountably many PMs}
	There are uncountably many non-isomorphic Puiseux monoids.
\end{prop}

\begin{proof}
	Consider the map $G \mapsto M_G := G \cap \qq_{\ge 0}$ sending each subgroup $G$ of $(\qq,+)$ to a Puiseux monoid. Clearly, $\gp(M_G) \cong G$. In addition, for all subgroups $G$ and~$G'$ of $(\qq,+)$, each monoid isomorphism between $M_G$ and $M_{G'}$ naturally extends to a group isomorphism between $G$ and $G'$. Hence our assignment sends non-isomorphic groups to non-isomorphic monoids. It follows from \cite[Corollary 85.2]{lF73} that there are uncountably many non-isomorphic rank-$1$ torsion-free abelian groups. As a result, there are uncountably many non-isomorphic Puiseux monoids.
\end{proof}

\smallskip
\subsection{Root and Complete Integral Closures}

Given a monoid $M$ with difference group $\gp(M)$, the sets
\begin{itemize}

	\item $\widetilde{M} := \big\{ x \in \gp(M) \mid nx \in M \ \text{for some} \ n \in \nn \big\}$ and

	\item $\widehat{M} := \big\{ x \in \gp(M) \mid \text{ there exists } c \in M \text{ such that } c + nx \in M \text{ for all } n \in \nn \big\}$
\end{itemize}
are called the \emph{root closure} and \emph{complete integral closure} of $M$, respectively. It is not hard to verify that $M \subseteq \widetilde{M} \subseteq \widehat{M} \subseteq \gp(M)$ for any monoid $M$. For Puiseux monoids, we give a complete description of these sets.

\begin{prop} \label{prop:closure of a PM}
	Let $M$ be a Puiseux monoid, and let $n = \gcd( \mathsf{n}(M^\bullet))$. Then the following statements hold.
	\begin{enumerate}
		\item \label{eq:equality of closures of a PM} $\widetilde{M} = \widehat{M} = \emph{\gp}(M) \cap \qq_{\ge 0}$.
		\vspace{2pt}
		\item $\widetilde{M} = n \langle 1/d \mid d \in \mathsf{d}(M^\bullet) \rangle.$
	\end{enumerate}
\end{prop}

\begin{proof}
	It is easy to verify that $\widehat{M} \subseteq \gp(M) \cap \qq_{\ge 0}$. So $\widetilde{M} \subseteq \widehat{M} \subseteq \gp(M) \cap \qq_{\ge 0}$. On the other hand, if $x \in \gp(M) \cap \qq_{>0}$, then it readily follows that $n \mathsf{d}(x)x \in M$. Therefore $\gp(M) \cap \qq_{\ge 0} \subseteq \widetilde{M}$. As a result,~(1) follows. 
	
	To prove~(2), fix $d \in \mathsf{d}(M^\bullet)$, and then take $k \in \nn$ such that $k/d \in M$. As $\gcd(k,nd) = n$, there exist $\alpha,\beta \in \nn_0$ satisfying that $n = \alpha k - \beta nd$. Then
	\[
		\frac{n}{d} = \alpha \frac{k}{d} - \beta n \in \gp(M).
	\]
	On the other hand, since $n \mid k$ and $(k/n) \frac{n}{d} = k/d \in M$, one finds that $n/d \in \widetilde{M}$. Thus, $n \langle 1/d \mid d \in \mathsf{d}(M^\bullet) \rangle \subseteq \widetilde{M}$. For the reverse inclusion we first verify that $\mathsf{d}(M^\bullet)$ is closed under taking positive divisors and least common multiples. Clearly, $\mathsf{d}(M^\bullet)$ is closed under taking positive divisors. To see that $\mathsf{d}(M^\bullet)$ is closed under taking least common multiples, take $q_1, q_2 \in M^\bullet$ and then set $d = \gcd( \mathsf{d}(q_1), \mathsf{d}(q_2))$ and $\ell = \text{lcm}(\mathsf{d}(q_1), \mathsf{d}(q_2))$. As $\gcd(\mathsf{n}(q_1), \mathsf{n}(q_2))$ is the greatest common divisor of $\mathsf{n}(q_1) \mathsf{d}(q_2)/d$ and $\mathsf{n}(q_2) \mathsf{d}(q_1)/d$, there exist $N, c_1, c_2 \in \nn_0$ such that
	\begin{align*}
		\frac{(N\ell + 1) \gcd(\mathsf{n}(q_1), \mathsf{n}(q_2))}{\ell}
		&= \frac{1}{\ell} \bigg( c_1 \frac{\mathsf{n}(q_1) \, \mathsf{d}(q_2)}d + c_2 \frac{\mathsf{n}(q_2) \, \mathsf{d}(q_1)}d \bigg)\\ 
			&= c_1 q_1 +  c_2  q_2 \in M.
	\end{align*}
	Since $\ell$ and $(N\ell + 1) \gcd(\mathsf{n}(q_1), \mathsf{n}(q_2))$ are relatively prime, $\ell \in \mathsf{d}(M^\bullet)$. Thus, $\mathsf{d}(M^\bullet)$ is closed under taking least common multiples, as desired.
	
	One can easily see that $M \subseteq n\langle 1/d \mid d \in \mathsf{d}(M^\bullet) \rangle$. Take $q \in \widetilde{M}\setminus M \subseteq \gp(M)$, and then take $q_1, q_2 \in M^\bullet$ such that $q = q_2 - q_1$. Clearly, $\mathsf{d}(q) \mid \text{lcm}(\mathsf{d}(q_1), \mathsf{d}(q_2))$. As $\mathsf{d}(M^\bullet)$ is closed under taking positive divisors and least common multiples, one finds that $\mathsf{d}(q) \in \mathsf{d}(M^\bullet)$. Hence $q = (\mathsf{n}(q)/n) \frac{n}{\mathsf{d}(q)} \in n \langle 1/d \mid d \in \mathsf{d}(M^\bullet) \rangle$, which implies the reverse inclusion.
\end{proof}

\begin{example} \label{ex:root closure of PMs} \hfill
	\begin{enumerate}
		\item For each $r \in \qq_{> 0}$, it is clear that $M = \{0\} \cup \qq_{\ge r}$ is a Puiseux monoid. Observe that $\gcd (\mathsf{n}(M^\bullet)) = 1$ and $\mathsf{d}(M^\bullet) = \nn$. Now one can use Proposition~\ref{prop:closure of a PM} to obtain that
		\[
			\widetilde{M} = \bigg\langle \frac{1}{d} \ \bigg{|} \ d \in \mathsf{d}(M^\bullet) \bigg\rangle = \bigg\langle \frac{1}{d} \ \bigg{|} \ d \in \nn \bigg\rangle = \qq_{\ge 0}.
		\]
		\item Now suppose that $p$ and $q$ are two distinct prime numbers, and consider the Puiseux monoid $M = \langle (p/q)^n \mid n \in \nn_0 \rangle$. Because $1 \in M$, it follows that $\gcd(\mathsf{n}(M^\bullet)) = 1$. In addition, it is clear that $\mathsf{d}(M^\bullet) = \{q^n \mid n \in \nn_0\}$. Now Proposition~\ref{prop:closure of a PM} ensures that
		\[
			\widetilde{M} = \bigg\langle \frac{1}{q^n} \ \bigg{|} \ n \in \nn_0 \bigg\rangle,
		\]
		which is the nonnegative cone of the localization of $\zz$ at the multiplicative set $\{q^n \mid n \in \nn_0\}$. Notice that the monoid $M$ is closed under multiplication and, therefore, it is a cyclic rational semiring; we will discuss the atomic structure of cyclic rational semirings in Section~\ref{sec:atomicity}.
	\end{enumerate}
\end{example}

A monoid $M$ is said to be \emph{root-closed} provided that $\widetilde{M} = M$. In addition, $M$ is called a \emph{Pr\"ufer monoid} if $M$ is the union of an ascending sequence of cyclic submonoids. Root-closed Puiseux monoids can be characterized in the following ways.

\begin{cor}
	For a Puiseux monoid $M$, the following statements are equivalent.
	\begin{enumerate}
		\item $M$ is root-closed.
		\vspace{2pt}
		
		\item $\emph{gp}(M) = M \cup -M$.
		\vspace{2pt}
		
		\item $M$ is a Pr\"ufer monoid.
	\end{enumerate}
\end{cor}

\begin{proof}
	The equivalence (1) $\Leftrightarrow$ (2) follows from Proposition~\ref{prop:closure of a PM}, while (1) $\Leftrightarrow$ (3) follows from~\cite[Theorem~13.5]{rG84}.
\end{proof}

We now characterize finitely generated Puiseux monoids in terms of their root closures and their sets of denominators.

\begin{prop} \label{prop:fg PM are NM}
	For a Puiseux monoid $M$ the following statements are equivalent.
	\begin{enumerate}
		\item $\widetilde{M} \cong (\nn_0,+)$.
		\vspace{2pt}
		
		\item $M$ is finitely generated.
		\vspace{2pt}
		
		\item $\mathsf{d}(M^\bullet)$ is finite.
		\vspace{2pt}
		
		\item $M$ is isomorphic to a numerical monoid.
	\end{enumerate}
\end{prop}

\begin{proof}
	To prove (1) $\Rightarrow (2)$, suppose that $\widetilde{M} \cong (\nn_0,+)$. Proposition~\ref{prop:closure of a PM} ensures that $\mathsf{d}(M^\bullet)$ is finite. Now if $\ell := \text{lcm} \, \mathsf{d}(M^\bullet)$, then $\ell M$ is a submonoid of $(\nn_0,+)$ that is isomorphic to $M$. Hence $M$ is finitely generated. To argue (2) $\Rightarrow$ (3), it suffices to notice that if $S$ is a finite generating set of $M$, then every element of $\mathsf{d}(M^\bullet)$ divides $\text{lcm} \, \mathsf{d}(S^\bullet)$. For (3) $\Rightarrow$ (4), let $\ell := \text{lcm} \, \mathsf{d}(M^\bullet)$. Then note that $\ell M$ is a submonoid of $(\nn_0,+)$ that is isomorphic to $M$. As a result, $M$ is isomorphic to a numerical monoid. To prove (4) $\Rightarrow$ (1), assume that $M$ is a numerical monoid and that $\gp(M)$ is a subgroup of $(\zz,+)$. By definition of $\widetilde{M}$, it follows that $\widetilde{M} \subseteq \nn_0$. On the other hand, the fact that $\nn_0 \setminus M$ is finite immediately implies that $\nn_0 \subseteq \widetilde{M}$. Consequently, $\widetilde{M} = (\nn_0,+)$. 
\end{proof}

\begin{cor} \label{cor:closure of a non-finitely generated PM is antimatter}
	Let $M$ be a Puiseux monoid.  Then $M$ is not finitely generated if and only if $\widetilde{M}$ is antimatter.
\end{cor}

\begin{proof}
	Suppose first that $M$ is not finitely generated. Set $n = \gcd(\mathsf{n}(M^\bullet))$. It follows from Proposition~\ref{prop:closure of a PM} that $\widetilde{M} = \langle n/d \mid d \in \mathsf{d}(M^\bullet) \rangle$. Fix $d \in \mathsf{d}(M^\bullet)$. Since $\mathsf{d}(M^\bullet)$ is an infinite set that is closed under taking least common multiples, there exists $d' \in \mathsf{d}(M^\bullet)$ such that $d'$ properly divides $d$. As a consequence, $n/d'$ properly divides $n/d$ in $\widetilde{M}$ and so $n/d \notin \mathcal{A}(\widetilde{M})$. As none of the elements in the generating set $\{n/d \mid d \in \mathsf{d}(M^\bullet)\}$ of $\widetilde{M}$ is an atom, $\widetilde{M}$ must be antimatter. The reverse implication is an immediate consequence of Proposition~\ref{prop:fg PM are NM}.
\end{proof}

One can use Corollary~\ref{cor:closure of a non-finitely generated PM is antimatter} to produce Puiseux monoids with no atoms.

\begin{example} \hfill
	\begin{enumerate}
		\item Take $r \in \qq_{> 0}$, and consider the Puiseux monoid $M = \{0\} \cup \qq_{\ge r}$. As $\mathsf{d}(M^\bullet)$ is not finite, it follows from Proposition~\ref{prop:fg PM are NM} that $M$ is not finitely generated. Then $\widetilde{M}$ is antimatter by Corollary~\ref{cor:closure of a non-finitely generated PM is antimatter}. We have already seen in Example~\ref{ex:root closure of PMs} that $\widetilde{M} = \qq_{\ge 0}$.
		\smallskip
		
		\item Consider the Puiseux monoid $M = \langle 1/p \mid p \in \pp \rangle$. As $0$ is a limit point of $M^\bullet$, the monoid $M$ is not finitely generated. Therefore it follows from Corollary~\ref{cor:closure of a non-finitely generated PM is antimatter} that $\widetilde{M}$ is an antimatter Puiseux monoid. Indeed, it is clear that $\gcd(\mathsf{n}(M^\bullet)) = 1$ and $\mathsf{d}(M^\bullet) = \{n \in \nn \mid n \ \text{is squarefree} \}$, so Proposition~\ref{prop:closure of a PM} guarantees that
		\[
			\widetilde{M} = \bigg\langle \frac{1}{n} \ \bigg{|} \ n \in \nn \ \text{is squarefree} \bigg\rangle.
		\]
		The atomicity of $M$ will be considered in Section~\ref{sec:atomicity}.
	\end{enumerate}
	
\end{example}

\smallskip
\subsection{Description of the Conductor}

Let $M$ be a monoid. The \emph{conductor} of $M$ is defined to be
\begin{equation} \label{eq:conductor}
	\mathfrak{c}(M) := \{ x \in \gp(M) \mid x + \widehat{M} \subseteq M \}.
\end{equation}
It is clear that $\mathfrak{c}(M)$ is a subsemigroup of the group $\gp(M)$. By Proposition~\ref{prop:closure of a PM}, the equality $\mathfrak{c}(M) = \{ x \in \gp(M) \mid x + \widetilde{M} \subseteq M \}$ holds when $M$ is a Puiseux monoid. This equality is more convenient for our purposes.

For a numerical monoid~$N$, the term ``conductor'' refers to the number $\mathfrak{f}(N) + 1$, where $\mathfrak{f}(N)$ is the Frobenius number of~$N$. As indicated in the following example, the conductor, as used in the context of numerical monoids, is the minimum of the conductor semigroup, as defined in~(\ref{eq:conductor}) and used in commutative semigroup theory.

\begin{example}
	Let $N$ be a numerical monoid, and let $\mathfrak{f}(N)$ be the Frobenius number of $N$. It follows from~(\ref{eq:difference group of a PM}) that $\gp(N) = \zz$. Therefore Proposition~\ref{prop:closure of a PM} guarantees that $\widetilde{N} = \nn_0$. For $n \in N$ with $n \ge \mathfrak{f}(N) + 1$, it is clear that $n + \widetilde{N} = n + \nn_0 \subseteq N$. On the other hand, for each $n \in \zz$ with $n \le \mathfrak{f}(N)$ the fact that $\mathfrak{f}(N) \in n + \widetilde{N}$ implies that $n + \widetilde{N} \nsubseteq N$. As a result,
	\begin{equation} \label{eq:Frobenius number and conductor}
		\mathfrak{c}(N) = \{ n \in \zz \mid n \ge \mathfrak{f}(N) + 1 \}.
	\end{equation}
	As the equality of sets~(\ref{eq:Frobenius number and conductor}) shows, the minimum of $\mathfrak{c}(N)$ is $\mathfrak{f}(N) + 1$, namely, the conductor number of $N$, as defined in the context of numerical monoids.

\end{example}

We conclude this subsection with a description of the conductor of a Puiseux monoid, which was recently established in~\cite{GGT19}. 

\begin{prop} \label{prop:conductor of a PM} 
	Let $M$ be a Puiseux monoid. Then the following statements hold.
	\begin{enumerate}
		\item If $M$ is root-closed, then $\mathfrak{c}(M) = \widetilde{M} = M$.

		\item If $M$ is not root-closed, then set $\sigma = \sup \, \widetilde{M} \setminus M$.
		\begin{enumerate} \label{part 2: conductor of PM}
			\item If $\sigma = \infty$, then $\mathfrak{c}(M) = \emptyset$.
			\item If $\sigma < \infty$, then $\mathfrak{c}(M) = M_{\ge \sigma}$.
		\end{enumerate}
	\end{enumerate}
\end{prop}

\begin{proof}
	To verify~(1), note that $\widetilde{M} \subseteq \qq_{\ge 0}$ implies that $\mathfrak{c}(M) \subseteq \qq_{\ge 0}$. As a result, $\mathfrak{c}(M) \subseteq \gp(M) \cap \qq_{\ge 0} = \widetilde{M}$. This, along with the obvious fact that $\widetilde{M} \subseteq \mathfrak{c}(M)$, implies~(1). 
	
	To show~(2), suppose that $M$ is not root-closed. It follows from Proposition~\ref{prop:closure of a PM} that $\mathfrak{c}(M) \subseteq \gp(M) \cap \qq_{\ge 0} = \widetilde{M}$. \\
	\indent \emph{Case (a)}: Take $\widetilde{x} \in \widetilde{M}$. Since $\widetilde{M} \setminus M$ is unbounded, there exists $\widetilde{x}_1 \in \widetilde{M} \setminus M$ such that $\widetilde{x}_1 > \widetilde{x}$. Then taking $\widetilde{y} := \widetilde{x}_1 - \widetilde{x} \in \widetilde{M}$, we can see that $\widetilde{x} + \widetilde{y} = \widetilde{x}_1 \notin M$. Therefore $\widetilde{x} \notin \mathfrak{c}(M)$. So we conclude that $\mathfrak{c}(M) = \emptyset$. \\
	\indent \emph{Case (b)}: As in the above paragraph, we can argue that no element in $\widetilde{M}_{< \sigma}$ is in $\mathfrak{c}(M)$ and then $\mathfrak{c}(M) \subseteq M_{\ge \sigma}$. For the reverse inclusion, take $x \in M_{\ge \sigma}$. If $\sigma \notin M$, then $x > \sigma$ and so $x + \widetilde{M} \subseteq (\widetilde{M} + \widetilde{M})_{> \sigma} \subseteq \widetilde{M}_{> \sigma} = M_{> \sigma} \subset M$. Thus, $x \in \mathfrak{c}(M)$ when $\sigma \notin M$. If $\sigma \in M$, then $\widetilde{M}_{\ge \sigma} = M_{\ge \sigma}$ and so $x + \widetilde{M} \subseteq (\widetilde{M} + \widetilde{M})_{\ge \sigma} \subseteq \widetilde{M}_{\ge \sigma} \subseteq M$. Therefore $x \in \mathfrak{c}(M)$ also when $\sigma \in M$. Hence $M_{\ge \sigma} \subseteq \mathfrak{c}(M)$.
\end{proof}

\begin{remark}
	With notation as in Proposition~\ref{prop:conductor of a PM}.\ref{part 2: conductor of PM}, although $\widetilde{M}_{> \sigma} = M_{> \sigma}$ holds, it can happen that $\widetilde{M}_{\ge \sigma} \neq M_{\ge \sigma}$. For instance, consider the Puiseux monoid $\{0\} \cup \qq_{>1}$.
\end{remark}

\smallskip
\subsection{Homomorphisms Between Puiseux Monoids}

As we are about to show, the homomorphisms between Puiseux monoids are those given by rational multiplication. 

\begin{prop} \label{prop:homomorphisms between PM}
	The homomorphisms between Puiseux monoids are given by rational multiplication.
\end{prop}

\begin{proof}
	Every rational-multiplication map is clearly a homomorphism. Suppose, on the other hand, that $\varphi \colon M \to M'$ is a homomorphism between Puiseux monoids. As the trivial homomorphism is multiplication by $0$, one can assume without loss of generality that $\varphi$ is nontrivial. Let $\{n_1, \dots, n_k\}$ be the minimal generating set of the additive monoid $N := M \cap \nn_0$. Since $\varphi$ is nontrivial, $k \ge 1$ and $\varphi(n_j) \neq 0$ for some $j \in \ldb 1,k \rdb$. Set $q = \varphi(n_j)/n_j$ and then take $r \in M^\bullet$ and $\alpha_1, \dots, \alpha_k \in \nn_0$ such that $\mathsf{n}(r) = \alpha_1 n_1 + \dots + \alpha_k n_k$. As $n_i \varphi(n_j) = \varphi(n_i n_j) = n_j \varphi(n_i)$ for every $i \in \ldb 1,k \rdb$,
	\[
		\varphi(r) = \frac 1{\mathsf{d}(r)} \varphi(\mathsf{n}(r)) = \frac 1{\mathsf{d}(r)} \sum_{i=1}^k \alpha_i \varphi(n_i) = \frac 1{\mathsf{d}(r)} \sum_{i=1}^k \alpha_i n_i \frac{\varphi (n_j)}{n_j} = rq.
	\]
	Thus, the homomorphism $\varphi$ is multiplication by $q \in \mathbb{Q}_{>0}$.
\end{proof}

\medskip
\section{Atomic Structure}
\label{sec:atomicity}

It is well known that in the class consisting of all monoids, the following chain of implications holds.
\begin{equation} \label{eq:monoid atomicity taxonomy}
	\textbf{UFM} \ \Rightarrow \ \big[ \textbf{HFM}, \ \textbf{FFM} \big] \ \Rightarrow \ \textbf{BFM} \ \Rightarrow \ \textbf{ACCP} \ \Rightarrow \ \textbf{atomic monoid}
\end{equation}
It is also known that, in general, none of the implications in~(\ref{eq:monoid atomicity taxonomy}) is reversible (even in the class of integral domains~\cite{AAZ90}). In this section, we provide various examples to illustrate that none of the above implications, except the first one, is reversible in the class of Puiseux monoids. We characterize the Puiseux monoids belonging to the first two classes of the chain of implications~(\ref{eq:monoid atomicity taxonomy}). For each of the last four classes, we find a family of Puiseux monoids belonging to such a class but not to the class right before.

\smallskip
\subsection{A Class of Atomic Puiseux Monoids}

We begin this section collecting a simple characterization of finitely generated Puiseux monoids in terms of the atomicity.

\begin{prop}
	A Puiseux monoid $M$ is finitely generated if and only if $M$ is atomic and $\mathcal{A}(M)$ is finite.
\end{prop}

\begin{proof}
	The direct implication follows from Proposition~\ref{prop:fg PM are NM}. The reverse implication is obvious because the atomicity of $M$ means that $M$ is generated by $\mathcal{A}(M)$.
\end{proof}

Corollary~\ref{cor:closure of a non-finitely generated PM is antimatter} yields, however, instances of nonfinitely generated Puiseux monoids containing no atoms. As the next example shows, for every $n \in \nn$ there exists a nonfinitely generated Puiseux monoid containing exactly $n$ atoms.

\begin{example}
	Let $m \in \nn$, and take distinct prime numbers $p$ and $q$ with $q > m$. Consider the Puiseux monoid $M = \big\langle \ldb m, 2m-1 \rdb \cup \{qp^{-m-i} \mid i \in \nn\} \big\rangle$. To verify that $\mathcal{A}(M) = \ldb m, 2m-1 \rdb$, write $a \in \ldb m, 2m-1 \rdb$ as
	\begin{equation} \label{eq:monoid with m atoms}
		a = a' + \sum_{n=1}^N \alpha_n \frac{q}{p^{m+n}},
	\end{equation}
	where $a' \in \{0\} \cup \ldb m, 2m-1 \rdb$ and $\alpha_n \in \nn_0$ for every $n \in \ldb 1,N \rdb$. After clearing denominators in both sides of~(\ref{eq:monoid with m atoms}), one finds that $q \mid a - a'$. Therefore $a = a'$ and $\alpha_1 =  \dots = \alpha_N = 0$, which implies that $a \in \mathcal{A}(M)$. Thus, $\ldb m, 2m-1 \rdb \subseteq \mathcal{A}(M)$. Clearly, $qp^{-m-i} \notin \mathcal{A}(M)$ for any $i \in \nn$. Hence $\mathcal{A}(M) = \ldb m, 2m-1 \rdb$, and so $|\mathcal{A}(M)| = m$. As $\mathsf{d}(M^\bullet)$ is not finite, it follows from Proposition~\ref{prop:fg PM are NM} that $M$ is not finitely generated.
\end{example}

Perhaps the class of nonfinitely generated Puiseux monoids that has been most thoroughly studied is that one consisting of cyclic rational semirings~\cite{CGG19}.
\smallskip

\noindent{\bf Notation.}\;  For $r \in \qq_{>0}$, we let $S_r$ denote the Puiseux monoid $\langle r^n \mid n \in \nn_0 \rangle$.
\smallskip

Although $S_r$ is indeed a cyclic rational semiring, we shall only be concerned here with its additive structure. The atomicity of $S_r$ was first studied in \cite[Section~6]{GG17} while several factorization aspects were investigated in~\cite{CGG19}. 

\begin{prop} \label{prop:atomic classification of multiplicative cyclic Puiseux monoids}
	For $r \in \qq_{> 0}$, consider the Puiseux monoid~$S_r$. Then the following statements hold.
	\begin{enumerate}
		\item If $r \ge 1$, then $S_r$ is atomic and \label{item:case r at least 1}
		\begin{itemize}
			\item either $r \in \nn$ and so $S_r = \nn_0$,
			
			\item or $r \notin \nn$ and so $\mathcal{A}(S_r) = \{r^n \mid n \in \nn_0\}$.
		\end{itemize}
		\item If $r < 1$, then \label{item:case r less than 1}
		\begin{itemize}
			\item either $\mathsf{n}(r) = 1$ and so $S_r$ is antimatter,
			
			\item or $\mathsf{n}(r) \neq 1$ and $S_r$ is atomic with $\mathcal{A}(S_r) = \{r^n \mid n \in \nn_0\}$.
		\end{itemize}
	\end{enumerate}
\end{prop}

\begin{proof}
	To argue~(\ref{item:case r at least 1}), suppose that $r \ge 1$. If $r \in \nn$, then it easily follows that $S_r = \nn_0$. Then we assume that $r \notin \nn$. Clearly, $\mathcal{A}(S_r) \subseteq \{r^n \mid n \in \nn_0\}$. To check the reverse inequality, fix $j \in \nn_0$ and write $r^j = \sum_{i=0}^N \alpha_i r^i$ for some $N \in \nn_0$ and coefficients $\alpha_0, \dots, \alpha_N \in \nn_0$. Because $(r^n)_{n \in \nn_0}$ is an increasing sequence, one can assume that $N \le j$. Then, after clearing denominators in $r^j = \sum_{i=0}^N \alpha_i r^i$ we obtain $N=j$ as well as $\alpha_j = 1$ and $\alpha_i = 0$ for every $i \neq j$. Hence $r^j \in \mathcal{A}(S_r)$ for every $j \in \nn_0$, yielding the second statement of~(\ref{item:case r at least 1}).
	
	Now suppose that $r < 1$. If $\mathsf{n}(r) = 1$, then $r^n = \mathsf{d}(r)r^{n+1}$ for every $n \in \nn_0$, and so~$S_r$ is antimatter, which is the first statement of~(\ref{item:case r less than 1}). Finally, suppose that $\mathsf{n}(r) > 1$. Fix $j \in \nn$, and notice that $r^i \nmid_{S_r} r^j$ for any $i < j$. Then write  $r^j = \sum_{i=j}^{j+k} \beta_i r^i$, for some $k \in \nn_0$ and $\beta_i \in \nn_0$ for every $i \in \ldb j, j+k \rdb$. Notice that $\beta_j \in \{0,1\}$. Suppose for a contradiction that $\beta_j = 0$. In this case, $k \ge 1$. Let $p$ be a prime dividing $\mathsf{n}(r)$, and let~$\alpha$ be the maximum power of $p$ dividing $\mathsf{n}(r)$. From $r^j = \sum_{i=j}^{j+k} \beta_i r^i$ one obtains
	\begin{equation} \label{eq:multiplicative cyclic 3}
		\alpha j = \pval\big( r^j \big) = \pval \bigg( \sum_{i=1}^k \beta_{j+i} r^{j+i} \bigg) \ge \min_{i \in \ldb 1,k \rdb} \big\{ \pval \big(\beta_{j+i} r^{j+i} \big) \big\} \ge \alpha(j+m),
	\end{equation}
	where $m = \min\{i \in \ldb 1, k \rdb \mid \beta_{j+i} \neq 0\}$. The inequality \eqref{eq:multiplicative cyclic 3} yields the desired contradiction. Hence $r^j \in \mathcal{A}(S_r)$ for every $j \in \nn_0$, yielding the second statement of~(\ref{item:case r less than 1}).
\end{proof}

\begin{cor}
	For each $r \in \qq \cap (0,1)$ with $\mathsf{n}(r) \neq 1$, the monoid $S_r$ is an atomic monoid that does not satisfy the ACCP.
\end{cor}

\begin{proof}
	Proposition~\ref{prop:atomic classification of multiplicative cyclic Puiseux monoids} guarantees the atomicity of $S_r$. To verify that $S_r$ does not satisfy the ACCP, consider the sequence of principal ideals $(\mathsf{n}(r)r^n + S_r)_{n \in \nn_0}$. Since
	\[
		\mathsf{n}(r)r^n = \mathsf{d}(r)r^{n+1} = (\mathsf{d}(r) - \mathsf{n}(r))r^{n+1} + \mathsf{n}(r)r^{n+1},
	\]
	$\mathsf{n}(r)r^{n+1} \mid_{S_r} \mathsf{n}(r)r^n$ for every $n \in \nn_0$. Therefore $(\mathsf{n}(r)r^n + S_r)_{n \in \nn_0}$ is an ascending chain of principal ideals. In addition, it is clear that such a chain of ideals does not stabilize. Hence $S_r$ does not satisfy the ACCP, which completes the proof.
\end{proof}

\smallskip
\subsection{A Class of ACCP Puiseux Monoids}

We proceed to present a class of ACCP Puiseux monoids containing a subclass of monoids that are not BFMs.

\begin{theorem} \label{thm:a class of ACCP monoids}
	Every submonoid of $\langle 1/p \mid p \in \pp \rangle$ satisfies the ACCP.
\end{theorem}

\begin{proof}
	It suffices to prove that $M = \langle 1/p \mid p \in \pp \rangle$ satisfies the ACCP. Let $(p_n)_{n \in \nn}$ be a strictly increasing sequence with underlying set $\pp$. One can readily verify that for all $x \in M$ there exist $k,n \in \nn_0$ and $\alpha_1, \dots, \alpha_k \in \nn_0$ such that
	\begin{align} \label{eq:decomposition existence}
	x = n + \sum_{j=1}^k \alpha_j \frac{1}{p_j},
	\end{align}
	where $\alpha_k \neq 0$ and $\alpha_j \in \ldb 0,p_j \rdb$ for every $j \in \ldb 1,k \rdb$. Let us now check that the sum decomposition in (\ref{eq:decomposition existence}) is unique. To do this, take $\ell, m \in \nn_0$ and $\beta_1, \dots, \beta_\ell \in \nn_0$ with $\beta_\ell \neq 0$ and $\beta_j \in \ldb 0, p_j \rdb$ for every $j \in \ldb 1,\ell \rdb$ such that
	\begin{align} \label{eq:decomposition uniqueness}
	n + \sum_{j=1}^k \alpha_j \frac{1}{p_j} = m + \sum_{j=1}^{\ell} \beta_j \frac{1}{p_j}.
	\end{align}
	After completing with zero coefficients, we may assume that $k = \ell$. For each $j \in \ldb 1,k \rdb$, one can isolate $(\alpha_j - \beta_j)/p_j$ in (\ref{eq:decomposition uniqueness}) and apply the $p_j$-adic valuation, to obtain that $p_j \mid \alpha_j - \beta_j$. This implies that $\alpha_j = \beta_j$ for each $j \in \ldb 1, k \rdb$. As a consequence, $n = m$, and the uniqueness of the decomposition in~(\ref{eq:decomposition existence}) follows.
	
	With notation as in~\eqref{eq:decomposition existence}, we set $n(x) := n$ and $s(x) := \alpha_1 + \dots + \alpha_k$. 
	Observe that if $x' \mid_{M_P} x$ for some $x' \in M_P$, then $n(x') \le n(x)$. In addition, observe that if $x' \in M$ divides $x$ in $M$ properly, then $n(x') = n(x)$ implies that $s(x') < s(x)$. As a consequence of these observations, one deduces that each sequence $(q_n)_{n \in \nn}$ in $M$ satisfying that $q_{n+1} \mid_{M} q_n$ for every $n \in \nn$ must stabilize. As a result, $M$ satisfies the ACCP.
\end{proof}

\begin{cor} \label{cor:ACCP that is not BFM}
	There are Puiseux monoids satisfying the ACCP that are not BFMs.
\end{cor}

\begin{proof}
	Consider the Puiseux monoid $M = \langle 1/p \mid p \in \pp \rangle$. We have seen in Theorem~\ref{thm:a class of ACCP monoids} that $M$ satisfies the ACCP. However, it is clear that $p \in \mathsf{L}(1)$ for every $p \in \pp$. As $|\mathsf{L}(1)| = \infty$, the monoid $M$ is not a BFM.
\end{proof}

\smallskip
\subsection{Bounded Factorization Monoids}

Our next goal is to find a large class of Puiseux monoids that are BFMs. We do so with the following result.

\begin{theorem} \label{thm:BF sufficient condition}
	Let $M$ be a Puiseux monoid. If $0$ is not a limit point of $M^\bullet$, then $M$ is a BFM and hence is atomic.
\end{theorem}

\begin{proof}
	It is clear that $\mathcal{A}(M)$ consists of those elements of $M^\bullet$ that cannot be written as the sum of two positive elements of $M$. Since $0$ is not a limit point of $M$, there exists $\epsilon > 0$ such that $\epsilon < x$ for all $x \in M^\bullet$. Now we show that $M = \langle \mathcal{A}(M) \rangle$. Take $x \in M^\bullet$. Since $\epsilon$ is a lower bound for $M^\bullet$, the element $x$ can be written as the sum of at most $\lfloor x/\epsilon \rfloor$ elements of $M^\bullet$. Take the maximum $m \in \nn$ such that $x = a_1 + \dots + a_m$ for some $a_1, \dots, a_m \in M^\bullet$. By the maximality of $m$, it follows that $a_i \in \mathcal{A}(M)$ for every $i \in \ldb 1,m \rdb$, which means that $x \in \langle \mathcal{A}(M) \rangle$. Hence $M$ is atomic. We have already noticed that every element $x$ in $M^\bullet$ can be written as a sum of at most $\lfloor x/\epsilon \rfloor$ atoms, i.e., $|\mathsf{L}(x)| \le \lfloor x/\epsilon \rfloor$ for all $x \in M$. Thus, $M$ is a BFM.
\end{proof}

The converse of Theorem~\ref{thm:BF sufficient condition} does not hold, as the next example illustrates. 

\begin{example} \label{ex:BF PM with 0 as a limit point}
	Let $(p_n)_{n \in \nn}$ and $(q_n)_{n \in \nn}$ be two strictly increasing sequences of prime numbers satisfying that $q_n > p_n^2$ for every $n \in \nn$. Then set $M := \big \langle \frac{p_n}{q_n} \ \big{|} \ n \in \nn \big \rangle$. It follows from~\cite[Corollary~5.6]{GG17} that $M$ is atomic, and it is easy to verify that $\mathcal{A}(M) = \{p_n/q_n \mid n \in \nn\}$. To argue that $M$ is indeed a BFM, take $x \in M^\bullet$ and note that since both sequences $(p_n)_{n \in \nn}$ and $(q_n)_{n \in \nn}$ are strictly increasing, there exists $N \in \nn$ such that $q_n \nmid \mathsf{d}(x)$ and $p_n > x$ for every $n \ge N$. As a result, if $z \in \mathsf{Z}(x)$, then none of the atoms in $\{p_n/q_n \mid n > N\}$ can appear in $z$. From this, one can deduce that $\mathsf{Z}(x)$ is finite. Then $\mathsf{L}(x)$ is finite for any $x \in M$, and so $M$ is a BFM. However, $q_n > p_n^2$ for every $n \in \nn$ implies that $0$ is a limit point of $M^\bullet$.
\end{example}

As we have seen in Corollary~\ref{cor:ACCP that is not BFM}, not every ACCP Puiseux monoid is a BFM. However, it was proved in~\cite[Theorem~3.4]{GGT19} that under a mild assumption on conductors, each of these atomic conditions is equivalent to not having~$0$ as a limit point.

\begin{theorem} \label{thm:strongly primary characterization when conductor non-empty}
	If $M$ is a nontrivial Puiseux monoid with nonempty conductor, then the following statements are equivalent.
	\begin{enumerate}
		\item $0$ is not a limit point of $M^\bullet$.
		\vspace{2pt}
		
		\item $M$ is a BFM.
		\vspace{2pt}
		
		\item $M$ satisfies the ACCP.
	\end{enumerate}
\end{theorem}

\begin{proof}
	The implications (1) $\Rightarrow$ (2) and (2) $\Rightarrow$ (3) follow from Theorem~\ref{thm:BF sufficient condition} and \cite[Corollary 1.3.3]{GH06}, respectively. Note that such implications do not depend on whether the conductor of $M$ is nonempty.
	
	We proceed to show that (3) $\Rightarrow$ (1). Suppose for a contradiction that $0$ is a limit point of $M^\bullet$. Since $M$ satisfies the ACCP, it must be atomic. Thus, there exists a sequence of atoms $(a_n)_{n \in \nn}$ such that $a_n < 1/2^n$ for every $n \in \nn$. Then the series $ \sum_{n=1}^\infty a_n$ converges to a certain limit $\ell \in (0,1)$ and, as a consequence, $s_n := \sum_{i=1}^n a_i \in M \cap (0,1)$ for every $n \in \nn$. Since the conductor of $M$ is nonempty, Proposition~\ref{prop:conductor of a PM} guarantees that the set $\widehat{M} \setminus M$ is bounded. Take $x \in M$ such that $x > 1 + \sup \widehat{M} \setminus M$. It follows from Proposition~\ref{prop:closure of a PM} that the inclusion $x - s_n \in \widehat{M}$ holds for every $n \in \nn$. As $M$ has nonempty conductor and $x - s_n > \sup \widehat{M} \setminus M$, Proposition~\ref{prop:conductor of a PM} ensures that $x - s_n \in M$. Consider the sequence of principal ideals $(x - s_n + M)_{n \in \nn}$ of~$M$. Since
	\[
		x - s_n = x - s_{n+1} + (s_{n+1} - s_n) = (x - s_{n+1}) + a_{n+1},
	\]
	the sequence $(x - s_n + M)_{n \in \nn}$ is an ascending chain of principal ideals of $M$. Because $x - s_{n+1} < x - s_n$, the ascending chain of principal ideals $(x - s_n + M)_{n \in \nn}$ does not stabilize. This contradicts that $M$ satisfies the ACCP, which completes the proof.
\end{proof}

As Corollary~\ref{cor:ACCP that is not BFM} and Example~\ref{ex:BF PM with 0 as a limit point} indicate, without the nonempty-conductor condition, none of the last two statements in Theorem~\ref{thm:strongly primary characterization when conductor non-empty} implies its predecessor. In addition, even inside the class of Puiseux monoids with nonempty conductor, neither being atomic nor being an FFM is equivalent to being a BFM (or satisfying the ACCP).

\begin{example}
	Consider the Puiseux monoid $M := \{0\} \cup \qq_{\ge 1}$. It is clear that the conductor of $M$ is nonempty. In addition, it follows from Theorem~\ref{thm:strongly primary characterization when conductor non-empty} that $M$ is a BFM. Note that $\mathcal{A}(M) = [1,2)$. However, $M$ is far from being an FFM; for instance, the formal sum $(1 + 1/n) + (x - 1 - 1/n)$ is a length-$2$ factorization in $\mathsf{Z}(x)$ for all $x \in (2,3] \cap \qq$ and $n \ge \big\lceil \frac{1}{x-2} \big\rceil$, which implies that $|\mathsf{Z}(x)| = \infty$ for all $x \in M_{>2}$.
\end{example}

\begin{example}
	Now consider the Puiseux monoid $M =  \langle 1/p \mid p \in \pp \rangle \cup \qq_{\ge 1}$. Since the monoid $\langle 1/p \mid p \in \pp \rangle$ is atomic by Theorem~\ref{thm:a class of ACCP monoids}, it is not hard to check that $M$ is also atomic. It follows from Proposition~\ref{prop:conductor of a PM} that $M$ has nonempty conductor. Since~$0$ is a limit point of~$M^\bullet$, Theorem~\ref{thm:strongly primary characterization when conductor non-empty} ensures that $M$ does not satisfy the ACCP.
\end{example}



\smallskip
\subsection{Finite Factorization Monoids}

Our next task is to introduce a class of Puiseux monoids that are FFMs. This class consists of all Puiseux monoids that can be generated by an increasing sequence of rationals.

\begin{definition}
	A Puiseux monoid $M$ is called \emph{increasing} (respectively \emph{decreasing}) if $M$ can be generated by an increasing (respectively decreasing) sequence. A Puiseux monoid is \emph{monotone} if it is increasing or decreasing.
\end{definition}

Not every Puiseux monoid is monotone, as the next example shows.

\begin{example} \label{ex:bounded PM that is neither decreasing nor increasing}
	Let $p_1, p_2, \dots$ be a strictly increasing enumeration of $\pp$. Consider the Puiseux monoid $M = \langle A \cup B \rangle$, where
	\[
		A = \bigg\{ \frac{1}{p_{2n}} \ \bigg{|} \ n \in \nn \bigg\} \ \text{ and } \ B = \bigg\{ \frac{p_{2n-1} - 1}{p_{2n-1}} \ \bigg{|} \ n \in \nn \bigg\}.
	\]
	It follows immediately that both $A$ and $B$ belong to $\mathcal{A}(M)$. So $M$ is atomic and $\mathcal{A}(M) = A \cup B$. Every generating set of $M$ must contain $A \cup B$ and so will have at least two limit points, namely, $0$ and $1$. Since every monotone sequence of rationals can have at most one limit point in the real line, we conclude that $M$ is not monotone.
\end{example}

The next proposition offers a first insight into the atomicity of increasing monoids.

\begin{prop} \label{prop:atoms of increasing monoids}
	Every increasing Puiseux monoid is atomic. Moreover, if $(r_n)_{n \in \nn}$ is an increasing sequence of positive rationals generating a Puiseux monoid $M$, then $\mathcal{A}(M) = \{r_n \mid r_n \notin \langle r_1, \dots, r_{n-1} \rangle\}$.
\end{prop}

\begin{proof}
	The fact that $M$ is atomic follows from observing that $r_1$ is a lower bound for~$M^\bullet$ and so $0$ is not a limit point of $M^\bullet$. To prove the second statement, set
	\[
		A = \{r_n \mid r_n \notin \langle r_1, \dots, r_{n-1} \rangle\}.
	\]
	Note that $A$ is finite if and only if $M$ is finitely generated, in which case it is clear that $A = \mathcal{A}(M)$. Then suppose that $A$ is not finite. List the elements of $A$ as a strictly increasing sequence, namely, $(a_n)_{n \in \nn}$. Note that $M = \langle A \rangle$ and $a_n \notin \langle a_1, \dots, a_{n-1} \rangle$ for any $n \in \nn$. Since $a_1 = \min M^\bullet$, we have that $a_1 \in \mathcal{A}(M)$. Take $n \in \ldb 2, |A| \rdb$. Because~$(a_n)_{n \in \nn}$ is a strictly increasing sequence and $a_n \notin \langle a_1,\dots, a_{n-1} \rangle$, one finds that~$a_n$ cannot be written as a sum of elements in $M$ in a nontrivial manner. Hence $a_n \in \mathcal{A}(M)$ for every $n \in \nn$, and one can conclude that $\mathcal{A}(M) = A$.
\end{proof}

Let us collect another characterization of finitely generated Puiseux monoids, now in terms of monotonicity.

\begin{prop}
	A nontrivial Puiseux monoid $M$ is finitely generated if and only if~$M$ is both increasing and decreasing.
\end{prop}

\begin{proof}
	The direct implication is obvious. For the reverse implication, suppose that~$M$ is a nontrivial Puiseux monoid that is increasing and decreasing. Proposition~\ref{prop:atoms of increasing monoids} implies that $M$ is atomic and, moreover, $\mathcal{A}(M)$ is the underlying set of an increasing sequence. Assume for a contradiction that $|\mathcal{A}(M)| = \infty$. Then $\mathcal{A}(M)$ does not contain a largest element. As $M$ is decreasing, there exists $D := \{d_n \mid n \in \nn \} \subset \qq_{> 0}$ such that $d_1 > d_2 > \cdots$ and $M = \langle D \rangle$. Let $m = \min \{n \in \nn \mid d_n \in \mathcal{A}(M)\}$, which must exist because $\mathcal{A}(M) \subseteq D$. Now the minimality of $m$ implies that $d_m$ is the largest element of $\mathcal{A}(M)$, which is a contradiction. Hence $\mathcal{A}(M)$ is finite. As $M$ is atomic, it must be finitely generated.
\end{proof}

\begin{definition}
	A Puiseux monoid is \emph{strongly increasing} if it can be generated by a sequence $(s_n)_{n \in \nn}$ such that $\lim_{n \to \infty} s_n = \infty$.
\end{definition}

Not every increasing Puiseux monoid is strongly increasing; to see this, consider the monoid $\langle \frac{p-1}{p} \mid p \in \pp \rangle$. We will show that the strongly increasing property is hereditary. First, we argue the following lemma.

\begin{lemma} \label{lem:no limit point implies strongly increasing subset}
	Let $R$ be an infinite subset of $\qq_{\ge 0}$. If $R$ does not contain any limit points, then it is the underlying set of an increasing and unbounded sequence.
\end{lemma}

\begin{proof}
	For any $r \in R$ and $S \subseteq R$, the interval $[0,r]$ must contain only finitely many elements of $S$; otherwise there would be a limit point of $S$ in $[0,r]$. Therefore every nonempty subset of $R$ has a minimum element. So the sequence $(r_n)_{n \in \nn}$ recursively defined by $r_1 = \min R$ and $r_n = \min R \! \setminus \! \{r_1, \dots, r_{n-1}\}$ is strictly increasing and has $R$ as its underlying set. Since $R$ is infinite and contains no limit points, the increasing sequence $(r_n)_{n \in \nn}$ must be unbounded. Hence $R$ is the underlying set of the increasing and unbounded sequence $(r_n)_{n \in \nn}$.
\end{proof}

\begin{prop} \label{prop:strongly increasing iff super increasing}
	A nontrivial Puiseux monoid $M$ is strongly increasing if and only if every submonoid of $M$ is increasing.
\end{prop}

\begin{proof}
	If $M$ is finitely generated, then the statement of the theorem follows immediately. So we will assume for the rest of this proof that $M$ is not finitely generated.
	
	For the direct implication, suppose that $M$ is strongly increasing. Let us start by verifying that $M$ does not have any real limit points. By Proposition~\ref{prop:atoms of increasing monoids}, the monoid~$M$ is atomic. As $M$ is atomic and nonfinitely generated, $|\mathcal{A}(M)| = \infty$. Let $(a_n)_{n \in \nn}$ be an increasing sequence whose underlying set is $\mathcal{A}(M)$. Since $M$ is strongly increasing and $\mathcal{A}(M)$ is an infinite subset contained in every generating set of $M$, the sequence $(a_n)_{n \in \nn}$ is unbounded. Therefore, for every $r \in \rr$, the interval $[0,r]$ contains only finitely many elements of $(a_n)_{n \in \nn}$. There is no loss in assuming that such elements are $a_1, \dots, a_k$ for some $k \in \nn$. Since $\langle a_1, \dots, a_k \rangle \cap [0,r]$ is finite, $M \cap [0,r]$ is also finite. Because $|[0,r] \cap M| < \infty$ for all $r \in \rr$, it follows that $M$ does not have any limit points in $\rr$.
	
	Now suppose that $N$ is a nontrivial submonoid of $M$. Being a subset of $M$, the monoid $N$ cannot have any limit points in $\rr$. Thus, by Lemma~\ref{lem:no limit point implies strongly increasing subset}, the set $N$ is the underlying set of an increasing and unbounded sequence of rationals. Hence $N$ is a strongly increasing Puiseux monoid.
	
	For the reverse implication, suppose that $M$ is not strongly increasing. We will check that, in this case, $M$ contains a submonoid that is not increasing. If $M$ is not increasing, then $M$ is a submonoid of itself that is not increasing. Suppose, therefore, that $M$ is increasing. By Proposition~\ref{prop:atoms of increasing monoids}, the monoid $M$ is atomic, and we can list its atoms increasingly. Let $(a_n)_{n \in \nn}$ be an increasing sequence with underlying set $\mathcal{A}(M)$. Because $M$ is not strongly increasing, there exists $\ell \in \rr_{> 0}$ that is the limit of the sequence $(a_n)_{n \in \nn}$. Since $\ell$ is a limit point of $M$, which is closed under addition, $2\ell$ and $3\ell$ must be limit points of~$M$. Let $(b_n)_{n \in \nn}$ and $(c_n)_{n \in \nn}$ be sequences in $M$ having infinite underlying sets such that $\lim_{n \to \infty} b_n = 2\ell$ and $\lim_{n \to \infty} c_n = 3\ell$. Furthermore, assume that for every $n \in \nn$,
	\begin{equation}
		|b_n - 2\ell| < \frac{\ell}4 \ \text{ and } \ |c_n - 3\ell| < \frac{\ell}4. \label{eq:decreasing is not hereditary}
	\end{equation}
	Let $N$ be the submonoid of $M$ generated by $A := \{b_n,c_n \mid n \in \nn\}$. Note that~$A$ contains at least two limit points. Let us verify that $N$ is atomic with $\mathcal{A}(N) = A$. The inequalities \eqref{eq:decreasing is not hereditary} immediately imply that $A$ is bounded from above by $3\ell + \ell/4$. On the other hand, proving that $\mathcal{A}(N) = A$ amounts to showing that $A$ and $A+A$ are disjoint. To verify this, it suffices to note that
	\begin{align*}
		\inf (A + A) &= \inf \big\{b_m + b_n, b_m + c_n, c_m + c_n \mid m,n \in \nn \big\} \\
						&\ge \min \bigg\{4\ell - \frac{\ell}{2}, \ 5\ell - \frac{\ell}{2}, \ 6\ell - \frac{\ell}{2} \bigg\} > 3\ell + \frac{\ell}{4} \ge \sup A.
	\end{align*}
	Thus, $\mathcal{A}(N) = A$. Since every increasing sequence has at most one limit point in $\rr$, the set $A$ cannot be the underlying set of an increasing rational sequence. As every generating set of $N$ contains $A$, we conclude that $N$ is not an increasing Puiseux monoid, which completes the proof.
\end{proof}

The next theorem provides a large class of Puiseux monoids that are FFMs.

\begin{theorem} \label{thm:increasing PM of Archimedean fields are FF}
	Every increasing Puiseux monoid is an FFM.
\end{theorem}

\begin{proof}
	Let $M$ be an increasing Puiseux monoid. Since $0$ is not a limit point of $M^\bullet$, Theorem~\ref{thm:BF sufficient condition} ensures that $M$ is a BFM. Suppose, by way of contradiction, that $M$ is not an FFM. Consider the set $S = \{x \in M : |\mathsf{Z}(x)| = \infty\}$. As $M$ is not an FFM, $S$ is not empty. In addition, $s = \inf S \neq 0$ because $0$ is not a limit point of $M^\bullet$. Note that~$M^\bullet$ contains a minimum element because $M$ is increasing. Set $m = \min M^\bullet$ and fix $\epsilon \in (0,m)$. Now take $x \in S$ such that $s \le x < s + \epsilon$. Each $a \in \mathcal{A}(M)$ appears in only finitely many factorizations of $x$; otherwise $x - a$ would be an element of $S$ satisfying $x - a < s$, which contradicts $s = \inf S$. Since $\mathsf{L}(x)$ is finite, there exists $\ell \in \mathsf{L}(x)$ such that the set
	\[
		Z = \{z \in \mathsf{Z}(x) : |z| = \ell\}
	\]
	has infinite cardinality. Fix $z_0 = a_1 \dots a_\ell \in \mathsf{Z}(x)$, where $a_1, \dots, a_\ell \in \mathcal{A}(M)$, and set $A = \max\{a_1, \dots, a_\ell\}$. As every atom appears in only finitely many factorizations in~$Z$ and $|Z| = \infty$, there exists $z_1 = b_1 \cdots b_\ell \in Z$, where $b_1, \dots, b_\ell \in \mathcal{A}(M)$ and $b_n > A$ for every $n \in \ldb 1,\ell \rdb$ (here we are using the fact that $\mathcal{A}(M)$ is the underlying set of an increasing sequence). But now, if $\pi \colon \mathsf{Z}(M) \to M$ is the factorization homomorphism of $M$,
	\[
		x = \pi(z_0) = \sum_{n=1}^\ell a_n \le A\ell < \sum_{n=1}^\ell b_n = \pi(z_1) = x,
	\]
	which is a contradiction. Hence $M$ is an FFM.
\end{proof}

On the other hand, the converse of Theorem~\ref{thm:increasing PM of Archimedean fields are FF} does not hold; the following example sheds some light upon this observation.

\begin{example} \label{ex:FF PM that is not increasing}
	Let $(p_n)_{n \in \nn}$ be a strictly increasing sequence of primes, and consider the Puiseux monoid defined as follows:
	\begin{equation} \label{eq:FF does not implies increasing}
		M = \langle A \rangle, \ \text{where} \ A = \bigg\{\frac{p_{2n}^2 + 1}{p_{2n}}, \frac{p_{2n+1} + 1}{p_{2n+1}} \ \bigg{|} \ n \in \nn \bigg\}.
	\end{equation}
	Since $A$ is an unbounded subset of $\rr$ having $1$ as a limit point, it cannot be increasing. In addition, as $\mathsf{d}(a) \neq \mathsf{d}(a')$ for all $a,a' \in A$ such that $a \neq a'$, every element of $A$ is an atom of $M$. As each generating set of $M$ must contain $A$ (which is not increasing), $M$ is not an increasing Puiseux monoid.
	
	To verify that $M$ is an FFM, fix $x \in M$ and then take $D_x$ to be the set of primes dividing $\mathsf{d}(x)$. Now choose $N \in \nn$ with $N > \max \{x, \mathsf{d}(x)\}$. For each $a \in A$ with $\mathsf{d}(a) > N$, the number of copies $\alpha$ of the atom $a$ appearing in any factorization of $x$ must be a multiple of $\mathsf{d}(a)$ because $\mathsf{d}(a) \notin D_x$. Then $\alpha = 0$; otherwise, we would obtain that $x \ge \alpha a \ge \mathsf{d}(a)a > \mathsf{d}(a) > x$. Thus, if an atom $a$ divides $x$ in $M$, then $\mathsf{d}(a) \le N$. As a result, only finitely many elements of $\mathcal{A}(M)$ divide $x$ in $M$, and so $|\mathsf{Z}(x)| < \infty$. Hence $M$ is an FFM that is not increasing.
\end{example}

\begin{remark}
	For an ordered field $F$, a \emph{positive monoid} of $F$ is an additive submonoid of the nonnegative cone of $F$. As for Puiseux monoids, a positive monoid is \emph{increasing} provided that it can be generated by an increasing sequence. Increasing positive monoids are FFMs~\cite[Theorem~5.6]{fG19}, but the proof of this general version of Theorem~\ref{thm:increasing PM of Archimedean fields are FF} is much more involved.
\end{remark}

\smallskip
\subsection{Factorial, Half-Factorial, and Other-Half-Factorial Monoids}

The only Puiseux monoid that is a UFM (or even an HFM) is, up to isomorphism, $(\nn_0,+)$. The following proposition formalizes this observation.

\begin{prop} \label{prop:HF PM characterization}
	For a nontrivial atomic Puiseux monoid $M$, the following statements are equivalent.
	\begin{enumerate}
		\item $M$ is a UFM.
		\vspace{2pt}
		
		\item $M$ is a HFM.
		\vspace{2pt}
		
		\item $M \cong (\nn_0,+)$.
		\vspace{2pt}
		
		\item $M$ contains a prime element.
	\end{enumerate}
\end{prop}

\begin{proof}
	Clearly, (3) $\Rightarrow$ (1) $\Rightarrow$ (2). To argue (2) $\Rightarrow$ (3), assume that $M$ is an HFM. Since~$M$ is an atomic nontrivial Puiseux monoid, $\mathcal{A}(M)$ is not empty. Let $a_1$ and~$a_2$ be two atoms of $M$. Then $z_1 := \mathsf{n}(a_2) \mathsf{d}(a_1) a_1$ and $z_2 :=  \mathsf{n}(a_1) \mathsf{d}(a_2) a_2$ are two factorizations of the element $\mathsf{n}(a_1) \mathsf{n}(a_2) \in M$. As $M$ is an HFM, it follows that $|z_1| = |z_2|$ and so $\mathsf{n}(a_2) \mathsf{d}(a_1) = \mathsf{n}(a_1)  \mathsf{d}(a_2)$. Then $a_1 = a_2$, which implies that $|\mathcal{A}(M)| = 1$. As a consequence, $M \cong (\nn_0,+)$.
	
	Because (3) $\Rightarrow$ (4) holds trivially, we only need to argue (4) $\Rightarrow$ (3). Fix a prime element $p \in M$ and take $a \in \mathcal{A}(M)$. Since $p \mid_M \mathsf{n}(p) \mathsf{d}(a) a$, one finds that $p \mid_M a$. This, in turn, implies that $a = p$. Hence $\mathcal{A}(M) = \{p\}$, and so $M \cong (\nn_0,+)$.
\end{proof}

\begin{example}
	For general monoids, the property of being an HFM is strictly weaker than that of being a UFM. For instance, the submonoid $\langle (1,n) \mid n \in \nn \rangle$ of $\nn_0^2$ is an HFM that is not a UFM (see \cite[Propositions~5.1 and~5.4]{fG19e} for more details).
\end{example}

A dual notion of being an HFM was introduced in~\cite{CS11} by Coykendall and Smith.

\begin{definition}
	An atomic monoid $M$ is an \emph{OHFM} (or an \emph{other-half-factorial monoid}) if for all $x \in M \setminus U(M)$ and $z, z' \in \mathsf{Z}(x)$ the equality $|z| = |z'|$ implies that $z = z'$.
\end{definition}

Clearly, every UFM is an OHFM. Although the multiplicative monoid of an integral domain is a UFM if and only if it is an OHFM~\cite[Corollary~2.11]{CS11}, OHFMs are not always UFMs or HFMs, even in the class of Puiseux monoids.

\begin{prop} \label{prop:OHF PM characterization}
	For a nontrivial atomic Puiseux monoid $M$, the following statements are equivalent.
	\begin{enumerate}
		\item $M$ is an OHFM.
		\vspace{2pt}
		
		\item $|\mathcal{A}(M)| \le 2$.
		\vspace{2pt}
		
		\item $M$ is isomorphic to a numerical monoid with embedding dimension $1$ or $2$.
	\end{enumerate}
\end{prop}

\begin{proof}
		To prove (1) $\Rightarrow$ (2), let $M$ be an OHFM. If $M$ is a UFM, then $M \cong (\nn_0,+)$, and we are done. Then suppose that $M$ is not a UFM. In this case, $|\mathcal{A}(M)| \ge 2$. Assume for a contradiction that $|\mathcal{A}(M)| \ge 3$. Take $a_1, a_2, a_3 \in \mathcal{A}(M)$ satisfying $a_1 < a_2 < a_3$. Let $d = \mathsf{d}(a_1) \mathsf{d}(a_2) \mathsf{d}(a_3)$, and set $a'_i = d a_i$ for each $i \in \ldb 1,3 \rdb$. Since $a'_1, a'_2$, and $a'_3$ are integers satisfying $a'_1 < a'_2 < a'_3$, there exist $m,n \in \nn$ such that
		\begin{equation} \label{eq:OHF}
			m(a'_2 - a'_1) = n(a'_3 - a'_2).
		\end{equation}
		Clearly, $z_1 := ma_1 + na_3$ and $z_2 := (m+n)a_2$ are two distinct factorizations in $\mathsf{Z}(M)$ satisfying $|z_1| = m+n = |z_2|$. In addition, after dividing both sides of the equality~(\ref{eq:OHF}) by $d$, one obtains $ma_1 + na_3 = (m+n)a_2$, which means that $z_1$ and $z_2$ are factorizations of the same element. However, this contradicts that $M$ is an OHFM. Hence $|\mathcal{A}(M)| \le 2$.
		
		To show that (2) $\Rightarrow$ (3), suppose that $|\mathcal{A}(M)| \le 2$. By~Proposition~\ref{prop:fg PM are NM}, $M$ is isomorphic to a numerical monoid $N$. As $|\mathcal{A}(M)| \le 2$, the embedding dimension of $N$ belongs to $\{1,2\}$, as desired.
		
		To argue that (3) $\Rightarrow$ (1), suppose that either $M \cong (\nn_0,+)$ or $M \cong \langle a, b \rangle$ for some $a, b \in \nn_{\ge 2}$ with $\gcd(a,b) = 1$. If $M \cong (\nn_0,+)$, then~$M$ is a UFM and, in particular, an OHFM. On the other hand, if $M \cong \langle a, b \rangle$, then it is an OHFM by~\cite[Example~2.13]{CS11}.
\end{proof}

\begin{example}
	There are Puiseux monoids that are FFMs but neither HFMs nor OHFMs. As a direct consequence  of Theorem~\ref{thm:increasing PM of Archimedean fields are FF} and Propositions~\ref{prop:HF PM characterization} and~\ref{prop:OHF PM characterization}, one finds that $\langle \frac{p-1}{p} \mid p \in \pp \rangle$ is one such monoid.
\end{example}

\section*{Acknowledgments}

While working on this project, the second author was supported by the NSF award DMS-1903069. The authors are grateful to Editor Susan Colley and two anonymous referees for many useful comments and suggestions that helped to improve the final version of this article.


	\noindent \footnotesize{{\bf Scott T. Chapman} is a Texas State University System Regents’ Professor and SHSU Distinguished Professor at Sam Houston State University in Huntsville, Texas. In December of 2016 he finished a five year appointment as Editor of the \textit{American Mathematical Monthly}. He is currently serving a three year term as Editor-in-Chief at \textit{Communications in Algebra}.  His editorial work, numerous publications in the area of non-unique factorizations, and years of directing REU Programs, led to his designation in 2017 as a Fellow of the American Mathematical Society.}

	\bigskip
	\noindent \footnotesize{{\bf Felix Gotti} received his Ph.D. degree from UC Berkeley in 2019, and he is currently an NSF Postdoctoral Fellow at the University of Florida. During the academic year 2018-2019, he was a research exchange scholar at Harvard University, where he completed his doctoral dissertation. Felix was also a research assistant at the University of Graz (Austria) in Summer 2019, working on arithmetic of monoids. His current research interest primarily lies in commutative algebra and geometric combinatorics.}
	
	\bigskip
	\noindent \footnotesize{{\bf Marly Gotti} holds the position of Senior Data Scientist at Biogen, where she contributes toward the development of drugs and therapies for neurological and neurodegenerative diseases. Along with her team at Biogen, Marly is currently investigating the effectiveness of a potential drug to treat Parkinson's disease. She obtained her Ph.D. degree in Pure Mathematics from the University of Florida. The theoretical flavor of her research lies in the area of factorization theory and arithmetic of commutative semigroups.}
	\bigskip

\end{document}